\documentclass[11pt]{article}
\usepackage{amsmath,amsfonts,amssymb,amsthm}
\usepackage{bm}
\usepackage{enumerate} 
\usepackage{indentfirst}
\usepackage{color}
\usepackage{float}
\usepackage{graphics}
\usepackage{epsfig}
\usepackage{amsbsy}
\usepackage{amsmath}
\usepackage{mathrsfs}
\usepackage{amsfonts}
\usepackage{comment}
\usepackage{empheq}
\usepackage{hyperref}
\hypersetup{colorlinks=true, 
	        linkcolor=red 
           }
\usepackage{authblk}
\AtBeginDocument{}
\allowdisplaybreaks
\newtheorem{theorem}{Theorem}[section]
\newtheorem{lemma}[theorem]{Lemma}
\newtheorem{proposition}[theorem]{Proposition}

\newtheorem{remark}[theorem]{Remark}

\def\R{{\mathbb{R}}}
\def\H{{I\!\!H}}

\def\CC{{\rm \kern.24em \vrule width.02em height1.4ex depth-.05ex
		\kern-.26emC}}
	
	\def\Xint#1{\mathchoice
		{\XXint\displaystyle\textstyle{#1}}%
		{\XXint\textstyle\scriptstyle{#1}}%
		{\XXint\scriptstyle\scriptscriptstyle{#1}}%
		{\XXint\scriptscriptstyle\scriptscriptstyle{#1}}%
		\!\int}
	\def\XXint#1#2#3{{\setbox0=\hbox{$#1{#2#3}{\int}$}
			\vcenter{\hbox{$#2#3$}}\kern-.5\wd0}}
	\def\dashint{\Xint-}
	
\newcommand{\pvint}{\mathop{\mathrlap{\pushpv}}\!\int}
\newcommand{\pushpv}{\mathchoice
	{\mkern19.5mu\rule[.6ex]{.5em}{.5pt}}
	{\mkern2.8mu\rule[.5ex]{.35em}{.8pt}}
	{\mkern2.5mu\rule[.50ex]{.7em}{.7pt}}
	{\mkern2mu\rule[.5ex]{.7em}{.5pt}}
}

\newcommand{\ppvint}{\mathop{\mathrlap{\pushppv}}\!\int}
\newcommand{\pushppv}{\mathchoice
	{\mkern28.9mu\rule[.6ex]{.5em}{.5pt}}
	{\mkern2.8mu\rule[.5ex]{.35em}{.8pt}}
	{\mkern2.5mu\rule[.50ex]{.7em}{.7pt}}
	{\mkern2mu\rule[.5ex]{.7em}{.5pt}}
}

\newcommand{\pppvint}{\mathop{\mathrlap{\pushpppv}}\!\int}
\newcommand{\pushpppv}{\mathchoice
	{\mkern31.9mu\rule[.6ex]{.5em}{.5pt}}
	{\mkern2.8mu\rule[.5ex]{.35em}{.8pt}}
	{\mkern2.5mu\rule[.50ex]{.7em}{.7pt}}
	{\mkern2mu\rule[.5ex]{.7em}{.5pt}}
}

\def\TagOnRight

\def\AA{{it I} \hskip-3pt{\tt A}}

\def\QQ{\rlap {\raise 0.4ex \hbox{$\scriptscriptstyle |$}} {\hskip -0.1em Q}}

\catcode`\@=11

\def\theequation{\@arabic{\c@section}.\@arabic{\c@equation}}

\renewcommand{\div}{{\mathrm{div}}}
\newcommand{\DT}{\mathbb{D}}

\newcommand{\HC}[2]{\mathcal{C}^{{#1},{#2}}}
\renewcommand{\L}[1]{L^{#1}(\Omega)}
\newcommand{\Lb}[1]{L^{#1}(\Gamma)}
\newcommand{\vL}[1]{\bm{L}^{#1}(\Omega)}

\renewcommand{\H}[1]{H^{#1}(\Omega)}

\newcommand{\W}[2]{W^{#1,#2}(\Omega)}
\newcommand{\vW}[2]{\bm{W}^{#1,#2}(\Omega)}
\newcommand{\Wfracb}[2]{W^{#1,#2}(\Gamma)}

\newcommand{\vu}{\bm{u}}
\newcommand{\vn}{\bm{n}}
\newcommand{\vt}{\bm{\tau}}

\begin{document}
\title{Uniform $W^{1,p}$ estimate for elliptic operator with Robin boundary condition in $\mathcal{C}^1$ domain}
	

\author[1]{C.~Amrouche \thanks{cherif.amrouche@univ-pau.fr}}
\author[2]{C.~Conca \thanks{cconca@dim.uchile.cl}}
\author[1,3]{A.~Ghosh \thanks{amrita.ghosh@univ-pau.fr}}
\author[4]{T.~Ghosh \thanks{tuhing@uw.edu}}

\affil[1]{LMAP, UMR CNRS 5142, Universit\'{e} de Pau et des Pays de l'Adour, France.}
\affil[2]{Departamento de Ingenier\'{i}a Matem\'{a}tica, Facultad de Ciencias F\'{i}sicas y Matem\'{a}ticas, Universidad de Chile.}
\affil[3]{Departamento de Matem\'{a}ticas, Universidad del Pa\'{i}s Vasco, 48940 Lejona, Spain.}
\affil[4]{Department of Mathematics, University of Washington, Seattle.}

\maketitle

\begin{abstract}
We consider the Robin boundary value problem  $\div (A\nabla u) = \div \bm{f}+F$ in $\Omega$, $\mathcal{C}^1$ domain, with $(A\nabla u - \bm{f})\cdot \vn + \alpha u =  g$ on $\Gamma$, where the matrix $A$ belongs to $VMO (\R^3) $,  and discover the uniform estimates on $\|u\|_{W^{1,p}(\Omega)}$, with $1 < p < \infty$,   independent on $\alpha$.  At the difference with the case $p = 2,$ which is simpler, we call here the weak reverse H\"older 
inequality. This estimates show that the solution of Robin problem converges strongly to the solution of Dirichlet (resp. Neumann) problem in corresponding spaces when the parameter $\alpha$ tends to $\infty$ (resp. $0$).
\end{abstract}

\section{Introduction and statement of main result}
\setcounter{equation}{0}
This paper is concerned with the second order elliptic operator of divergence form with Robin boundary condition. In a bounded domain (open, connected set) $\Omega$ in $\mathbb{R}^n$ with $\bm{f}\in \vL{p}, F\in \L{r(p)}$ and $g\in \Wfracb{-\frac{1}{p}}{p}$, consider the following problem
\begin{equation}
\label{robin1}
\left\{
\begin{aligned}
\mathcal{L} u &= \div \bm{f}+F \quad &&\text{ in } \Omega,\\
(A\nabla u - \bm{f})\cdot \vn + \alpha u &= g \quad &&\text{ on } \Gamma
\end{aligned}
\right.
\end{equation}
where
\begin{equation}
\label{10}
\mathcal{L} =  \div (A\nabla)
\end{equation}
with $A(x) = (a_{ij}(x))$ is an $3\times 3$ matrix with real-valued, bounded, measurable entries satisfying the following uniform ellipticity condition
\begin{equation*}
\mu |\xi|^2 \le A(x) \xi \cdot \xi \le \frac{1}{\mu} |\xi|^2 \quad \text{ for all } \,\,\, \xi,\ x \in \R^{n} \,\,\, \text{ and some } \mu >0. 
\end{equation*}
Here $\bm{n}$ is the outward unit normal vector on the boundary.


We want to study the well-posedness of the problem (\ref{robin1}), precisely, the existence, uniqueness of weak solution of (\ref{robin1}) in $W^{1,p}(\Omega)$ for any $p\in(1,\infty)$ and the bound on the solution, uniform in $\alpha$. Assuming $\alpha\ge 0$ a constant or a smooth function, the proof of existence of a unique solution provided $A\in VMO (\R^3) $  uses Neumann regularity results for elliptic problems; the interested reader is referred to \cite{DK} for details. The case $\alpha \le 0$ corresponds to the so-called Steklov eigenvalue problem. A recent survey on this topic can be seen in \cite{GP} and the references therein. That being said, our main interest in this work is to obtain precise estimates on the solution, in particular uniform estimates in $\alpha$.

Note that, formally, $\alpha = \infty$ corresponds to the Dirichlet boundary condition whereas $\alpha =0$ gives the Neumann boundary condition. In both Dirichlet and Neumann cases, we have the classical $W^{1,p}$ estimate of the solution.
And so for the Robin problem as follows:
\begin{equation*}
\|u\|_{\W{1}{p}} \le C(\alpha) \left( \|\bm{f}\|_{\vL{p}} + \|F\|_{\L{r(p)}} + \|g\|_{\Wfracb{-\frac{1}{p}}{p}}\right),
\end{equation*}
where $C(\alpha)$ depends also on $p$ and on $ \Omega$. Such well-posedness results on Robin boundary value problem for arbitrary domains can be found, for example, in \cite{DD}. But the continuity constant depends on $\alpha$ whereas the constant in Dirichlet (and Neumann) estimate has no $\alpha$. So it is natural to expect 
we may obtain $\alpha$-independent bound of the solution of problem (\ref{robin1}). That is,
if we let $\alpha$ tend to $\infty$, we show rigorously that we  get back the solution of the Dirichlet problem. The case when $\alpha$ goes to $0$ is relatively easier to handle (though not trivial) assuming the compatibility condition of the Neumann problem.

The purpose of this article is to estimate the continuity constant $C(\alpha)$ uniformly with respect to $\alpha$. Among the vast literature on Robin boundary value problem and various related questions to study, we did not find any reference concerning the question of \textit{behavior of the solution on the parameter $\alpha$} in the existing literature so far, even for Laplacian. 

One of the main motivation comes from the Stokes  (and also the Navier-Stokes) problem with Navier slip boundary condition
\begin{equation}
\label{SS}
\begin{cases}
-\nu\Delta\vu +\nabla \pi=\bm{f} ,\quad \div\;\vu=0 \ &\text{ in $\Omega$}\\
\vu\cdot\vn=0, \quad 2\left[(\DT\vu )\vn \right]_{\vt}+\alpha\vu_{\vt}=\bm{0} \ &\text{ on $\Gamma$}.
\end{cases}
\end{equation}
where the function $\alpha$ refers to the \textit{friction coefficient}. Here, to understand the behavior of the solution with respect to the friction coefficient is an important question to study (see \cite{Masmoudi2}, \cite{Pal}). Obtaining bound uniform in $\alpha$ in this case leads to study the limit problem: the solution of the Navier-Stokes equations with Navier boundary condition converges strongly in $\vW{1}{p}$ to the solution of the Navier-Stokes-Dirichlet problem as $\alpha$ tends to $\infty$ (see \cite{AG}). This observation may further enable us to recover informations concerning the solutions of the Navier-Stokes problem with Dirichlet boundary condition. Observe that the above Stokes system reduces to the problem (\ref{robin1}) in the simplest case, replacing the Stokes operator by Laplacian and the Navier boundary condition by Robin. To work with the full Navier-Stokes system with the complicated boundary condition was at the beginning quite cumbersome, thus we concentrated on the simpler scalar version. Surprisingly we found that this is itself an interesting question and still difficult to answer.

Here is our main result.
Throughout this work, the following assumption on $\alpha$ will be considered which we do not mention each time:
\begin{equation}
\label{alpha}
\alpha \in \Lb{t(p)} \quad \text{ and }\quad \alpha \ge \alpha_*>0 \quad \text{ on }\Gamma
\end{equation}
where $t(p)$ defined by
\begin{equation}\label{def_exponent_tp_alpha}
\begin{cases}
t(p) = 2 & \text{if\quad} p=2\\
t(p) = 2+\varepsilon & \text{if\quad} \frac{3}{2}\leq p\leq 3, p \neq 2\\
t(p)= \frac{2}{3} \max\{p,p^\prime\}+\varepsilon & \text{otherwise}
\end{cases}
\end{equation}
where $\varepsilon>0$ is arbitrary, satisfies $t(p) = t(p')$.

Also let $F\in \L{r(p)}$ where
\begin{equation}
\label{def:exponent_rp}
r(p)=
\begin{cases}\frac{3p}{p+3} & \text{if\quad}p>\frac{3}{2}\\
\text{any arbitrary real number}>1 & \text{if\quad}p=\frac{3}{2}\\
1 & \text{if\quad}p<\frac{3}{2}.
\end{cases}
\end{equation}

\begin{theorem}
	\label{T2}
Let $\Omega$ be a $\mathcal{C}^1$ bounded domain in $\R^3$, $p\in(1,\infty)$, $\bm{f}\in\vL{p}$, $F\in \L{r(p)}$, $g\in \Wfracb{-\frac{1}{p}}{p}$ and $\alpha \in \Lb{t(p)}$.
Suppose that the coefficients of the operator $\mathcal{L}$, defined in (\ref{10}), are symmetric and in $VMO(\R^3)$. Then the solution $u\in\W{1}{p}$ of (\ref{robin1}) satisfies the following estimate:
	\begin{equation}
	\label{Lp1}
	\|u\|_{\W{1}{p}} \le C_p(\Omega,\alpha_*) \left( \|\bm{f} \|_{\vL{p}} + \|F\|_{\L{r(p)}} + \|g\|_{\Wfracb{-\frac{1}{p}}{p}}\right) 
	\end{equation}
	where the constant $C_p(\Omega,\alpha_*)>0$ is independent of $\alpha$.
\end{theorem}

Notice that, with above estimate result, we immediately get that the solution of the Robin problem (\ref{robin1}) converges strongly to the solution of Dirichlet boundary problem in the corresponding spaces as $\alpha$ goes to $\infty$. To prove the above theorem, we first obtain the result for $F = 0$, $g=0$ and $p>2$ and then for $p<2$ using duality argument; And finally for $F\ne 0$, $g\ne 0$. Essentially we want to utilise the $\alpha$-independent $L^2$ gradient estimate (which follows from the variational formulation) to yield $L^p$ gradient estimate. The main tool in the proof for $p>2$ is a weak reverse H\"older inequality (wRHI) for gradient satisfied by the solution of the homogeneous problem, shown in Theorem \ref{T3}. Note that for Lipschitz domain, the weak reverse H\"older inequality is only true for certain values of $p$, even for Dirichlet boundary condition. It was first proved by Giaquinta \cite[Proposition 1.1, Chapter V]{giaquinta} in the case of Dirichlet condition, on smooth domain and for Laplace operator which follows from an argument by Gehring \cite{gehring}. 
wRHI in the case of $B(x,r)\subset \Omega$ follows from the classical interior estimate for harmonic functions. 
But in the case when $x\in \Gamma$, some suitable boundary H\"older estimate is required. In the present paper, to treat the operator in divergence form with $VMO$-coefficients, we use an approximation argument from the constant coefficient operator case, found in \cite{CP}. In the case of Neumann problem and for general second order elliptic operator, the proof of wRHI has been done in \cite[section 4]{geng} in Lipschitz domain; Whereas the sketch of the proof for Neumann problem in smooth domain has been given in \cite[p. 914]{kenig-shen}.

We obtain the similar result for $H^s$-bound (on Lipschitz domain) for $s\in (0,\frac{1}{2})$ in Theorem \ref{Hs} and $W^{2,p}$-estimate (on $\HC{1}{1}$ domain) in Theorem \ref{strong}.

\section{Related results and Proof of Theorem \ref{T2}}	
\setcounter{equation}{0}
To prove Theorem \ref{T2}, we start with studying the existence result. Only the case $n=3$ has been discussed here for the sake of clarity but all the results are true for $n=2$ as well  and the exact same proofs follow with the necessary modifications.

\begin{theorem}[{\bf Existence result in \boldmath$\W{1}{p}, p\ge 2$}]
\label{T0}
Let $\Omega$ be a $\mathcal{C}^{1}$ bounded domain in $\R^3$ and $p \ge 2$. Suppose that the coefficients of the operator $\mathcal{L}$, defined in (\ref{10}), are symmetric and in $VMO(\R^3)$. Then for any $\bm{f}\in\vL{p}$, $F\in\L{r(p)}$ and $g\in \Wfracb{-\frac{1}{p}}{p}$, there exists a unique solution $u\in \W{1}{p}$ of Problem (\ref{robin1}).
\end{theorem}

\begin{remark}
\rm{For $p=2$, $\Omega$ Lipschitz is sufficient to show the existence of solution $u\in \H1$}.
\end{remark}

\begin{proof}
It is trivial to see that $u\in\W{1}{p}$ is a solution of (\ref{robin1}) iff $u\in\W{1}{p}$ satisfies the following variational formulation:
\begin{equation}
\label{E}
\forall\varphi\in \W{1}{p'}, \quad \int\displaylimits_{\Omega}{A\nabla u\cdot \nabla \varphi} + \int\displaylimits_{\Gamma}{\alpha u \ \varphi}
= \int\displaylimits_\Omega{\bm{f}\cdot \nabla \varphi} - \int\displaylimits_\Omega {F \varphi}\,  + \left\langle  g, \varphi \right\rangle _{\Gamma}
\end{equation}
where $\left\langle \cdot,\cdot \right\rangle _{\Gamma}$ denotes the duality between $\Wfracb{-\frac{1}{p}}{p}$ and $ \Wfracb{\frac{1}{p}}{p'}$. Note that the boundary integral $\int\displaylimits_{\Gamma} \alpha u\ \varphi$ is well defined. For $p=2$, the bilinear form
$$\forall \ u,\varphi\in \H1, \quad a(u,\varphi) = \int\displaylimits_{\Omega}{A\nabla u\cdot \nabla \varphi} + \int\displaylimits_{\Gamma}{\alpha u \ \varphi}$$
is clearly continuous. Also, due to the ellipticity hypothesis on $A(x)$ and by Friedrich's inequality and the assumption $\alpha\ge \alpha_* > 0$ on $\Gamma$, we may have
$$a(u,u) = \int\displaylimits_{\Omega}{A \nabla u\cdot \nabla u} + \int\displaylimits_{\Gamma}{\alpha |u|^2}\ge C(\alpha_*, \mu) \ \|u\|^2_{\H1} $$
which shows that the bilinear form is coercive on $\H1$. And the right hand side of (\ref{E}) defines an element in the dual of $\H1$.
Thus, by Lax-Milgram lemma, there exists a unique $u\in\H1$ satisfying (\ref{E}). So we obtain the existence of a unique solution of (\ref{robin1}) in $\H1$.

Now for $p>2$, since $\vL{p}\hookrightarrow \vL{2}, \L{r(p)}\hookrightarrow \L{6/5}$,$\Wfracb{-\frac{1}{p}}{p} \hookrightarrow H^{-\frac{1}{2}}(\Gamma)$ and $\Lb{t(p)}\hookrightarrow \Lb{2}$, there exists a unique $u\in\H1$ solving (\ref{robin1}). It remains to show that $u\in\W{1}{p}$.

{\bf (i) $2<p\le 3$.} Since $u\in\H1 \hookrightarrow\Lb{4}$ and $\alpha\in\Lb{2+\varepsilon}$, we have $\alpha u\in\Lb{q_1}$ where $\frac{1}{q_1}=\frac{1}{4}+\frac{1}{2+\varepsilon}$. But using the Sobolev embedding $\Lb{q_1}\hookrightarrow\Wfracb{-\frac{1}{p_1}}{p_1}$ with $p_1=\frac{3}{2}q_1$ (since $q_1 > \frac{4}{3}$)\\ 
$$\text{i.e.}\quad \frac{1}{p_1}=\frac{2}{3}\left(\frac{1}{4}+\frac{1}{2+\varepsilon}\right),$$
Neumann regularity result (cf. \cite{DK}) implies $u\in\W{1}{p_1}$ since $\Omega$ is $\mathcal{C}^{1}$.
If $p_1\geq p$, we are done.
Otherwise, $u\in\W{1}{p_1}$.
Hence, $u\in\Lb{s_1}$ where $$\frac{1}{s_1}=\frac{1}{p_1}-\frac{1-\frac{1}{p_1}}{2}=\frac{3}{2p_1}-\frac{1}{2}$$\\
as $p_1<p \leq 3$. Then $\alpha u \in\Lb{q_2}$ where $\frac{1}{q_2}=\frac{1}{s_1}+\frac{1}{2+\varepsilon}$.
But, $\Lb{q_2}\hookrightarrow\Wfracb{-\frac{1}{p_2}}{p_2}$ with $p_2=\frac{3}{2}q_2$ \textit{i.e.}
$$ \frac{1}{p_2}=\frac{2}{3}\left(\frac{1}{4}+\frac{1}{2+\varepsilon}-\frac{1}{2}+\frac{1}{2+\varepsilon}\right)
=\frac{2}{3}\left(\frac{2}{2+\varepsilon}-\frac{1}{2}+\frac{1}{4}\right).
$$
If $p_2\geq p$, then as before, we have $u\in\W{1}{p} $. Otherwise, $u\in\W{1}{p_2}$.
Proceeding similarly, we get $u\in\W{1}{p_{k+1}}$ with
$$\frac{1}{p_{k+1}}=\frac{2}{3}\left(\frac{k+1}{2+\varepsilon}-\frac{k}{2}+\frac{1}{4}\right).$$
(where in each step, we assumed that $p_k<3$).
Now choosing $k=\lfloor\frac{1}{\varepsilon}-\frac{1}{2}\rfloor+1$ such that $p_{k+1}\geq3\geq p$ (where $\lfloor a\rfloor$ stands for the greatest integer less than or equal to $a$), we obtain $u\in\W{1}{p}$.\\
\\
{\bf (ii) $p>3$.} From the previous case, we obtain $u\in\W{1}{3}$ which gives $u\in\Lb{q}$ for all $1<q<\infty$.
But $\alpha\in\Lb{\frac{2}{3}p+\varepsilon}$ implies $\alpha u\in\Lb{\frac{2}{3}p}\hookrightarrow\Wfracb{-\frac{1}{p}}{p}$.
Therefore, using same reasoning as before, from the Neumann regularity result, we get $u\in\W{1}{p}$ .
\hfill
\end{proof}

Next we discuss the estimate of the solution of problem (\ref{robin1}) for $p>2$ with $F=0$ and $g=0$, independent of $\alpha$.
\begin{theorem}[{\bf \boldmath$\W{1}{p}$ estimate, $p\ge 2$ with RHS $\bm{f}$}]
\label{T1}
Let $\Omega$ be a $\mathcal{C}^1$ bounded domain in $\R^3$, $p\ge 2$  and $\bm{f}\in\vL{p}$. Suppose that the coefficients of the operator $\mathcal{L}$, defined in (\ref{10}), are symmetric and in $VMO(\R^3)$. Then the solution $u\in\W{1}{p}$ of (\ref{robin1}) with $F=0$ and $g=0$, satisfies the following estimate:
\begin{equation}
\label{Lp}
\|u\|_{\W{1}{p}} \le C_p(\Omega,\alpha_*) \ \|\bm{f} \|_{\vL{p}}
\end{equation}
where the constant $C_p(\Omega,\alpha_*)>0$ is independent of $\alpha$.
\end{theorem}

The proof of the above theorem is similar to that of Neumann problem \cite{geng}, once we have the wRHI. Since $\Omega$ is $\mathcal{C}^1$, there exists some $r_0>0$ such that for any $x_0\in\Gamma$, there exists a coordinate system $(x',x_3)$ which is isometric to the usual coordinate system and a $\mathcal{C}^1$ function $\psi : \R^2 \rightarrow \R$ so that,
$$
B(x_0,r_0)\cap \Omega = \left\lbrace (x',x_3)\in B(x_0,r_0) : x_3 > \psi(x')\right\rbrace
$$
and
$$
B(x_0,r_0)\cap \Gamma = \left\lbrace (x',x_3)\in B(x_0,r_0) : x_3 = \psi(x')\right\rbrace .
$$

In some places, we may write $B$ instead of $B(x,r)$ where there is no ambiguity and $aB := B(x, a r)$ for $a>0$. Also we use the usual notation to denote the average for any integrable function $f$ on a domain $\omega$:
$$
\dashint\displaylimits_{ \omega}{f} := \frac{1}{|\omega|}\int\displaylimits_{\omega}{f} := \overline{f}.
$$

We first prove the following weak reverse H\"{o}lder inequality for some $p=2+\varepsilon$, $\varepsilon>0$ whose proof is straight forward but this is not sufficient to deduce Theorem \ref{T1}.

\begin{lemma}
\label{Lrhi}
Let $\Omega$ be a $\mathcal{C}^1$ bounded domain in $\R^3$ and $\mathcal{L}$ be the operator defined in (\ref{10}). For any $B(x,r)$ with the property that $0<r<\frac{r_0}{8}$ and either $B(x,2r)\subset \Omega$ or $ x\in\Gamma$, the following weak Reverse H\"{o}lder inequalities hold: for some $\varepsilon>0$,\\
(i) if $B(x,2r) \subset \Omega$,
\begin{equation}
\label{rhi}
\left( \pvint\displaylimits_{ \ B(x,r)} {|\nabla v|^{2+\varepsilon}}\right) ^{1/2+\varepsilon} \le C \left( \pvint\displaylimits_{ \,\,B(x,2r)} {|\nabla v|^2}\right) ^{1/2} 
\end{equation}
whenever $v\in H^1(B(x,2r))$ satisfies $\mathcal{L} v = 0 \textup{ in } B(x,2r)$.\\
(ii) if $x\in\Gamma$,
\begin{equation}
\label{rhi1}
\left( \ppvint\displaylimits_{ \ B(x,r)\cap \Omega} {(|v|^2+|\nabla v|^2)^{\frac{2+\varepsilon}{2}}} \right)^{1/2+\varepsilon} \le C \left( \ppvint\displaylimits_{B(x,2r)\cap \Omega} {(|v|^2+|\nabla v|^2)}\right)^{1/2}
\end{equation}
whenever $v\in H^1(B(x,2r)\cap \Omega)$ satisfies
\begin{equation}
\label{erhi}
\begin{cases}
\mathcal{L} v &= 0 \quad  \textup{ in } \ B(x,2r)\cap\Omega\\
A\nabla v\cdot \vn + \alpha v &= 0 \quad \textup{ on } \ B(x,2r)\cap\Gamma.
\end{cases}
\end{equation}
The constants $C>0$ in the above estimates are independent of $\alpha$.
\end{lemma}

\begin{proof}
The proof of the weak Reverse H\"{o}lder inequality for Robin problem follows the similar argument as for the Dirichlet problem, established in \cite{giaquinta}.\\

{\bf case(i) :} $2B\subset \Omega$.\\
Since $v$ satisfies the equation $\div (A(x)\nabla) v =0$ in $2B$, we can have the following Caccioppoli inequality,
\begin{equation*}
\int\displaylimits_{B}{|\nabla v|^2} \le \frac{C}{r^2}\int\displaylimits_{2B}{|v-\overline{v}|^2}, \qquad \overline{v} = \frac{1}{|2B|} \int\displaylimits_{2B}{v}
\end{equation*}
for some constant $C>0$ independent of $\alpha$. Now using the following Sobolev-Poincar\'{e} inequality, for any $w \in \W{1}{p}$, $p>1$,
\begin{equation*}
\|w-\overline{w}\|_{\L{p^*}} \le C \|\nabla w\|_{\L{p}}, \qquad \overline{w} = \frac{1}{|\Omega|} \int\displaylimits_{\Omega}{w}
\end{equation*}
where $p^*$ is the Sobolev exponent, we obtain,
\begin{equation*}
\int\displaylimits_{B}{|\nabla v|^2} \le \frac{C}{r^2}\left( \ \int\displaylimits_{2B}{|\nabla v|^{\tilde{q}}}\right) ^{2/\tilde{q}}
\end{equation*}
with $\tilde{q}=6/5$ (this value comes from the dimension $n=3$). Upon normalizing both sides, we can write,
\begin{equation*}
\left( \frac{1}{r^3}\int\displaylimits_{B}{|\nabla v|^2}\right) ^{1/2} \le C \left( \frac{1}{r^3} \int\displaylimits_{2B}{|\nabla v|^{\tilde{q}}}\right) ^{1/\tilde{q}}.
\end{equation*}	
Here note that in $\R^3$, $|B| = c r^3$. Then setting $g=|\nabla v|^{\tilde{q}}$ and $q = 5/3 = 2/\tilde{q}$, we have,
\begin{equation*}
 \frac{1}{r^3}\int\displaylimits_{B} {g^{q}} \le C \left( \frac{1}{r^3}\int\displaylimits_{2B} {g}\right)^{q}.
\end{equation*}
Hence, \cite[Proposition 1.1]{giaquinta} with $f=0$ and $\theta=0$ implies, for some $\varepsilon>0$,
\begin{equation*}
\left( \ \dashint\displaylimits_{B}{|\nabla v|^{2+\varepsilon}}\right) ^{1/2+\varepsilon} \le C \left( \ \dashint\displaylimits_{2B}{|\nabla v|^{2}}\right) ^{1/2}.
\end{equation*}
	
{\bf case(ii) :} $x\in\Gamma$.\\
We follow the same path of the above interior estimate. First we want to prove a Caccioppoli type inequality for the problem (\ref{erhi}) up to the boundary. For that, let $\eta \in C_c^{\infty}(2B)$ be a cut-off function such that
$$
0\le \eta \le 1, \quad \eta \equiv 1 \ \text{ on } B \quad \text{ and } \quad |\nabla \eta| \le \frac{C}{r} .
$$
Now multiplying (\ref{erhi}) by $\eta ^2 v$ and integrating by parts, we get,
\begin{equation*}
\int\displaylimits_{2B\cap \Omega} {A\nabla v \cdot \nabla (\eta ^2 v)} + \int\displaylimits_{\partial (2B \cap \Omega)}{\alpha \eta ^2 v^2} = 0
\end{equation*}
which yields,
\begin{equation*}
\mu \int\displaylimits_{2B\cap \Omega}{\eta ^2 |\nabla v|^2 } +\int \displaylimits_{2B\cap \Gamma}{\alpha \eta ^2 v^2} \le \int\displaylimits_{2B\cap \Omega}{\eta ^2 A(x)\nabla v \cdot \nabla v} +\int \displaylimits_{2B\cap \Gamma}{\alpha \eta ^2 v^2} = - 2\int\displaylimits_{2B\cap \Omega}{\eta v \nabla v\cdot \nabla \eta}.
\end{equation*}
Using Cauchy's inequality on the right hand side, we obtain,
\begin{equation*}
\int\displaylimits_{2B\cap \Omega}{|\nabla v|^2 \eta ^2} + \int \displaylimits_{2B\cap \Gamma}{\alpha \eta ^2 v^2} \le 2 \left[  \frac{1}{4}\int\displaylimits_{2B\cap \Omega}{\eta^2 |\nabla v|^2} + 4 \int\displaylimits_{2B\cap \Omega} {v^2 |\nabla \eta|^2}\right] .
\end{equation*}
Simplifying the above estimate gives
\begin{equation*}
\int\displaylimits_{2B\cap \Omega}{|\nabla v|^2 \eta ^2} + \int \displaylimits_{2B\cap \Gamma}{\alpha \eta ^2 v^2} \le C \int\displaylimits_{2B\cap \Omega} {v^2 |\nabla \eta|^2},
\end{equation*}
which yields the Caccioppoli-type inequality, up to the boundary,
\begin{equation}
\label{caccio}
\int\displaylimits_{B\cap \Omega}{|\nabla v|^2 } +\int \displaylimits_{B\cap \Gamma}{\alpha v^2} \le \int\displaylimits_{2B\cap \Omega}{|\nabla v|^2 \eta ^2} + \int \displaylimits_{2B\cap \Gamma}{\alpha \eta ^2 v^2} \le \frac{C}{r^2} \int\displaylimits_{2B\cap \Omega} {v^2} .
\end{equation}
But we also have,
\begin{equation*}
\|v\|^2_{H^1(B\cap\Omega)} \le C \left( \ \int\displaylimits_{B\cap \Omega}{|\nabla v|^2 } +\int \displaylimits_{B\cap \Gamma}{ v^2}\right) \le C(\alpha_*) \left( \ \int\displaylimits_{B\cap \Omega}{|\nabla v|^2 } +\int \displaylimits_{B\cap \Gamma}{\alpha v^2}\right).
\end{equation*}
Hence, using (\ref{caccio}), we obtain,
\begin{equation*}
\int\displaylimits_{B\cap \Omega}{(|v|^2+|\nabla v|^2)} \le \frac{C(\alpha_*)}{r^2} \int\displaylimits_{2B\cap \Omega} {|v|^2} \le \frac{C(\alpha_*)}{r^2} \left( \ \int\displaylimits_{2B\cap \Omega}{(|v|^2+|\nabla v|^2)^{\tilde{q}/2}}\right) ^{2/\tilde{q}}
\end{equation*}
with $\tilde{q} = 6/5$ so that $(\tilde{q})^* = 2$. Thus,
\begin{equation*}
\begin{aligned}
\frac{1}{r^3}\int\displaylimits_{B\cap \Omega}{(|v|^2+|\nabla v|^2) } &\le \frac{C(\alpha_*)}{r^5} \left( \ \int\displaylimits_{2B\cap \Omega} {(|v|^2+|\nabla v|^2)^{\tilde{q}/2}}\right) ^{2/\tilde{q}}\\
&= C(\alpha_*) \left( \frac{1}{r^3} \int\displaylimits_{2B\cap \Omega} {(|v|^2+|\nabla v|^2)^{\tilde{q}/2}}\right) ^{2/\tilde{q}}.
\end{aligned}
\end{equation*}
Now setting,
\begin{equation*}
g(y) =
\begin{cases}
(|v|^2+|\nabla v|^2)^{\tilde{q}/2} \quad &\text{ if } \ y\in 2B\cap \Omega \\
 0 \quad &\text{ if } \ y\in 2B\setminus\Omega
\end{cases}
\end{equation*}
and $q = 2/\tilde{q}$, we obtain,
\begin{equation*}
 \frac{1}{r^3}\int\displaylimits_{B} {g^{q}} \le C(\alpha_*) \left( \frac{1}{r^3}\int\displaylimits_{2B} {g}\right)^{q} .
\end{equation*}
Once again \cite[Proposition 1.1]{giaquinta} with $f=0$ and $\theta =0$ implies, for some $\varepsilon>0$,
\begin{equation*}
\left( \ \dashint\displaylimits_{B} {g^{q+\varepsilon}} \right)^{1/q+\varepsilon} \le C \left( \ \dashint\displaylimits_{2B} {g^{q}}\right)^{1/q}
\end{equation*}
\textit{i.e.}
\begin{equation*}
\left( \pvint\displaylimits_{\ \ \ B\cap \Omega} {(|v|^2+|\nabla v|^2)^{(q+\varepsilon)\tilde{q}/2}} \right)^{1/q+\varepsilon} \le C \left( \pvint\displaylimits_{\ \,\, 2B\cap \Omega} {(|v|^2+|\nabla v|^2)}\right)^{\tilde{q}/2}
\end{equation*}
or equivalently, for some $s>2$,
\begin{equation*}
\left( \pvint\displaylimits_{\ \ \ B\cap \Omega} {(|v|^2+|\nabla v|^2)^{s/2}} \right)^{1/s} \le C \left( \pvint\displaylimits_{\ \,\,2B\cap \Omega} {(|v|^2+|\nabla v|^2)}\right)^{1/2}.
\end{equation*}
\hfill
\end{proof}
	
Next we prove wRHI for all $p>2$. For that we state the following boundary H\"{o}lder estimate for $\mathcal{L}$ under Robin boundary condition.
	
\begin{theorem}
		Let $\Omega$ be a $\mathcal{C}^1$ bounded domain in $\R^3$, $p>1$ and $\gamma \in (0,1)$. Suppose that the operator $\mathcal{L}$ defined in (\ref{10}) has constant and symmetric coefficients and
		\begin{equation*}
		\begin{cases}
		\mathcal{L} v &= 0 \quad  \textup{ in } \ B(Q,r)\cap\Omega\\
		A\nabla v\cdot \vn  + \alpha v &= 0 \quad \textup{ on } \ B(Q,r)\cap\Gamma
		\end{cases}
		\end{equation*}
		for some $Q\in \Gamma$ and $0<r<r_0$, Then for any $x,y \in B(Q,r/2)\cap \Omega$,
		\begin{equation}
		\label{bhe}
		|v(x)-v(y)|\le C \left( \frac{|x-y|}{r}\right) ^\gamma \left( \pppvint\displaylimits_{ \ \,\, B(Q,r)\cap\Omega}{|v|^p} \right) ^{1/p}
		\end{equation}
		where $C>0$ depends only on $\Omega, p$ and the ellipticity constant $\mu$, but independent of $\alpha$.
	\end{theorem}
	
	\begin{proof}
		Follows from classical regularity theory (for example, see \cite[Theorem 8.27]{GT}).
		\hfill
	\end{proof}
	
	Now the weak reverse H\"{o}lder inequality for any $p>2$ is proved in the case of constant coefficients.
	
	\begin{lemma}
		\label{Lrhi.}
		Let $\Omega$ be a $\mathcal{C}^1$ bounded domain in $\R^3$ and $p\ge 2$. Suppose that $\mathcal{L}$, defined in (\ref{10}), has constant and symmetric coefficients. Then for any $B(x,r)$ with the property that $0<r<\frac{r_0}{8}$ and either $B(x,2r)\subset \Omega$ or $ x\in\Gamma$, the following weak Reverse H\"{o}lder inequalities hold:\\
		(i) if $B(x,2r) \subset \Omega$,
		\begin{equation}
		\label{rhi.}
		\left( \pvint\displaylimits_{\,\,\,B(x,r)} {|\nabla v|^p}\right) ^{1/p} \le C \left( \ \pvint\displaylimits_{B(x,2r)} {|\nabla v|^2}\right) ^{1/2} 
		\end{equation}
		whenever $v\in H^1(B(x,2r))$ satisfies $\mathcal{L} v = 0 \textup{ in } B(x,2r)$.\\
		(ii) if $x\in\Gamma$,
		\begin{equation}
		\label{rhi1.}
		\left( \ppvint\displaylimits_{\,\,\,B(x,r)\cap \Omega} {|\nabla v|^p +|v|^p} \right)^{1/p} \le C \left( \ \ppvint\displaylimits_{B(x,2r)\cap \Omega} {|\nabla v|^2 + |v|^2}\right)^{1/2}
		\end{equation}
		whenever $v\in H^1(B(x,2r)\cap \Omega)$ satisfies
		\begin{equation*}
		\begin{cases}
		\mathcal{L} v &= 0 \quad  \textup{ in } \ B(x,2r)\cap\Omega\\
		A\nabla v\cdot \vn + \alpha v &= 0 \quad \textup{ on } \ B(x,2r)\cap\Gamma (\text{ if } x \in \Gamma).
		\end{cases}
		\end{equation*}
		The constant $C>0$ at most depends on $\Omega, p$ and the ellipticity constant $\mu$.
	\end{lemma}
	
	\begin{proof}
		Since $A$ is symmetric and positive definite, by a change of coordinate system, we may assume that $\mathcal{L} = \Delta$ (although we may consider the full operator and all the results hold true as well).
		
		The proof we will follow has been used for elliptic equations with Neumann boundary condition in \cite{kenig-shen}, just after the statement of Theorem 4.1.
		
		{\bf case(i) :} $B(x_0,2r)\subset \Omega$.\\
		The weak reverse H\"{o}lder inequality (\ref{rhi.}) holds for any $p\ge 2$, by the following well-known interior estimates for Harmonic functions, even when $\Omega$ is Lipschitz:
		\begin{equation*}
		\underset{B(x_0,r)}{\sup} \ |\nabla v|\le C \left( \ppvint\displaylimits_{\ \ \,B(x_0,2r)} {|\nabla v|^2}\right) ^{1/2}.
		\end{equation*}
		
		{\bf case(ii) :} $x_0\in\Gamma$.\\
		From the interior gradient estimate for harmonic function, we can write (eg. see \cite[Lemma 1.10]{HL})
		\begin{equation*}
		|\nabla v(x)|\le \frac{3}{\delta(x)} \ \underset{B(x,c\delta(x))}{\sup}|v|
		\end{equation*}
		for any $x\in B(x_0,r)\cap \Omega$ where $\delta(x) = d(x,\Gamma)$ and $c>0$ is chosen such that $B(x,2c\delta(x))\subsetneq B(x_0,2r)\cap\Omega$. From \cite[Remark 1.19]{HL}, we may then write
		\begin{equation*}
		|\nabla v(x)|\le \frac{C}{\delta(x)} \left( \pppvint\displaylimits_{ \ B(x, c\delta(x))}{|v|^2}\right) ^{1/2}.
		\end{equation*}
		Now for fixed $y\in B(x,2c\delta(x))$, let $u(x) = v(x) - v(y)$. Then $\mathcal{L}u = 0$ in $B(x,2c\delta(x))$ and thus we may write from the above argument,
		\begin{align*}
		|\nabla u(x)|\le \frac{C}{\delta(x)} \left( \pppvint\displaylimits_{ \ B(x,c\delta(x))}{|u|^2}\right) ^{1/2}
		\end{align*}
		which gives, along with the boundary H\"{o}lder estimate (\ref{bhe}),
		\begin{align}
		|\nabla v(x)| &\le \frac{C}{\delta(x)} \left( \pppvint\displaylimits_{ \ B(x,c\delta(x))}{|v(z) - v(y)|^2\mathrm{d}z}\right) ^{1/2}\nonumber\\
		& = \frac{C}{\delta(x)^{1+\frac{3}{2}}} \left( \ \int\displaylimits_{B(x,c\delta(x))}{|v(z) - v(y)|^2\mathrm{d}z}\right) ^{1/2}\nonumber\\
		& \le \frac{C}{\delta(x)^{1+\frac{3}{2}}} \left( \ \int\displaylimits_{B(x,2c\delta(x))}\left( \frac{|z-y|}{r}\right) ^{2\gamma}\left( \ \pppvint\displaylimits_{B(x_0,2r)\cap\Omega}{|v|^2}\right) \mathrm{d}z\right) ^{1/2}\nonumber\\
		& \le \frac{C}{\delta(x)^{1+\frac{3}{2}}} \left( \ \pppvint\displaylimits_{B(x_0,2r)\cap\Omega}{|v|^2}\right)^{1/2} \frac{1}{r^\gamma} \left( \ \int\displaylimits_{B(x,2c\delta(x))}{|z-y|^{2\gamma}\mathrm{d}z }\right)^{1/2}.\label{0}
		\end{align}
		Let us now calculate the last integral in the last inequality. Substituting $w = \frac{z-y}{4c\delta(x)}$, we get
		\begin{align*}
		\int\displaylimits_{B(x,2c\delta(x))}{|z-y|^{2\gamma}\mathrm{d}z } & \le C \int\displaylimits_{B(0,1)} {w^{2\gamma} (\delta(x))^{2\gamma +3} \mathrm{d}w}\\
		& = C \ (\delta(x))^{2\gamma +3} \int\displaylimits_{0}^{1}\int\displaylimits_0^{2\pi} {r^{2\gamma} r^2 \,\,\mathrm{d}r \,\,\mathrm{d}\theta}\\
		& = C \ (\delta(x))^{2\gamma +3} \ \frac{2\pi}{2\gamma +3} \le C_\gamma (\delta(x))^{2\gamma +3}.
		\end{align*}
		Plugging the value of the above integral in (\ref{0}), along with the Sobolev inequality, we then obtain
		\begin{align*}
		|\nabla v(x)| & \le \frac{C_\gamma}{(\delta(x))^{1+\frac{3}{2}}} \left( \ \pppvint\displaylimits_{B(x_0,2r)\cap\Omega}{|v|^2}\right)^{1/2} \frac{1}{r^\gamma} (\delta(x))^{\gamma + \frac{3}{2}}\\
		& = C_\gamma \ \frac{(\delta(x))^{\gamma -1}}{r^\gamma} r^{-3/2} \left( \ \int\displaylimits_{B(x_0,2r)\cap\Omega}{|v|^2}\right) ^{1/2}\\
		& \le C_\gamma \ \frac{(\delta(x))^{\gamma -1}}{r^\gamma} r^{1-3/2} \left( \ \int\displaylimits_{B(x_0,2r)\cap\Omega}{|v|^6}\right)^{1/6}\\
		& \le C_\gamma \left( \frac{r}{\delta(x)}\right) ^{1-\gamma}\left( \ \pppvint\displaylimits_{B(x_0,2r)\cap\Omega}{|\nabla v|^2+|v|^2}\right)^{1/2}.
		\end{align*} 
		Since $\gamma \in (0,1)$ is arbitrary, we thus have,
		\begin{equation*}
		|\nabla v(x)| \le C_\gamma \left( \frac{r}{\delta(x)}\right) ^{\gamma}\left( \ \pppvint\displaylimits_{B(x_0,2r)\cap\Omega}{|\nabla v|^2+|v|^2}\right)^{1/2}.
		\end{equation*}	
		Finally it yields choosing $\gamma$ so that $p\gamma <1$,
		\begin{equation*}
		\left(  \pppvint\displaylimits_{ \ B(x_0,r)\cap \Omega}{|\nabla v|^p}\right) ^{1/p} \le C_p \left( \ \pppvint\displaylimits_{B(x_0,2r)\cap\Omega}{|\nabla v|^2+|v|^2}\right)^{1/2}.
		\end{equation*}
		This completes the proof.
		\hfill
	\end{proof}
	
	A function $f$ in $BMO(\R^n)$ is said to be in $VMO(\R^n)$ if
	\begin{equation*}
	\underset{r\to 0}{\lim} \ \underset{x_0\in \R^n}{\sup} \ \frac{1}{r^n}\int\displaylimits_{B(x_0,r)} {|f-\overline{f} \ | \ \mathrm{d}x} =0
	\end{equation*}
	where $\overline{f} = \frac{1}{|B(x_0,r)|} \int\displaylimits_{B(x_0,r)}{f}$.\\
	
	\vspace{3pt}
	
	To treat the elliptic operator with $VMO$ coefficients, we prove the following approximation argument, found in \cite{CP}.
	
	\begin{lemma}
		\label{Lapprox}
		Let $\Omega$ be a $\mathcal{C}^1$ bounded domain in $\R^3$. Suppose that the coefficients of operator $\mathcal{L}$, defined in (\ref{10}), are symmetric and in $VMO(\R^3)$. Then there exists a function $h(r)$ and some constants $C>0, c>0$ with the following properties:\\
		\textit{i)} $\lim\limits_{r\to 0} h(r) =0$;\\
		\textit{ii)} for any $v\in H^{1}$ solution of
		\begin{equation}
		\label{1.}
		\begin{cases}
		\mathcal{L} v &= 0 \quad  \textup{ in } \ B(x,8r)\cap\Omega\\
		A\nabla v\cdot \vn  + \alpha v &= 0 \quad \textup{ on } \ B(x,8r)\cap\Gamma
		\end{cases}
		\end{equation}
		with $x\in \overline{\Omega}$ and $0<r<c r_0$, there exists a function $w \in W^{1,p}(B(x,r)\cap \Omega)$ such that for any $p>2$,
		\begin{equation}
		\label{4}
		\left( \pppvint\displaylimits_{ \ \  B(x,r)\cap \Omega}{|\nabla v - \nabla w|^2+|v-w|^2}\right) ^{1/2} \le h(r) \left( \ \pppvint\displaylimits_{\ B(x,8r)\cap \Omega}{|\nabla v|^2+|v|^2} \right) ^{1/2}
		\end{equation}
		
		\begin{equation}
		\label{3}
		\left( \pppvint\displaylimits_{ \ \ B(x,r)\cap \Omega}{|\nabla w|^p +|w|^p}\right) ^{1/p} \le C \left( \ \pppvint\displaylimits_{\,\,B(x,8r)\cap \Omega}{|\nabla v|^2 +|v|^2 } \right) ^{1/2},
		\end{equation}
		where the constant $C>0$ depends at most on $\Omega, p, \alpha_*, \mu$ and $A$.
	\end{lemma}
	
	\begin{proof}
		Let us fix $x_0\in \overline{\Omega}$ and $0<r < c r_0$ where $0<c<<1$ is such that Lemma \ref{Lrhi.} can be applied suitably. Let $v\in H^1(B(x_0, 8r)\cap\Omega)$ be a weak solution of (\ref{1.}). Consider
		\begin{equation}
		\label{2}
		\begin{cases}
		\div (B\nabla w)= 0 \quad &\textup{ in } \ B(x_0,4r)\cap\Omega\\
		(B\nabla w)\cdot \vn + \alpha w = (A\nabla v)\cdot \vn + \alpha v  \quad &\textup{ on } \ \partial (B(x_0,4r)\cap\Omega)
		\end{cases}
		\end{equation}
		where $B=(b_{ij})_{1\le i,j\le 3}$ are the constants given by
		\begin{equation*}
		b_{ij} = \frac{1}{|B(x_0,8r)|}\int\displaylimits_{B(x_0,8r)} {a_{ij}(x) \ \mathrm{d}x} .
		\end{equation*}
		So, $w\in H^1(B(x_0,4r)\cap \Omega)$ is a weak solution of (\ref{2}) if for all $\varphi \in H^1(B(x_0,4r)\cap\Omega)$,
		\begin{equation*}
		\int\displaylimits_{B(x_0,4r)\cap \Omega}{B\nabla w\cdot \nabla \varphi} +\int\displaylimits_{\partial (B(x_0,4r)\cap\Omega)}{\alpha w \ \varphi}=  \int\displaylimits_{B(x_0,4r)\cap \Omega}{A\nabla v\cdot \nabla \varphi} +\int\displaylimits_{\partial (B(x_0,4r)\cap\Omega)}{\alpha v \ \varphi}.
		\end{equation*}
		The existence of $w\in H^1(B(x_0,4r)\cap \Omega)$ follows immediately from the regularity of $v$. It then follows
		\begin{equation*}
		\int\displaylimits_{B(x_0,4r)\cap \Omega}{B\nabla (v-w)\cdot \nabla \varphi} + \int\displaylimits_{\partial (B(x_0,4r)\cap\Omega)} {\alpha (v-w)\varphi}= \int\displaylimits_{B(x_0,4r)\cap \Omega}{(B-A)\nabla v\cdot \nabla \varphi}.
		\end{equation*}
		Next we show that $w$ satisfies estimates (\ref{4}) and (\ref{3}).
		
		To see (\ref{4}), choosing $\varphi = v-w$, by ellipticity and Cauchy inequality, we obtain
		\begin{equation*}
		\begin{aligned}
		& \mu \int\displaylimits_{B(x_0,4r)\cap \Omega} {|\nabla (v-w)|^2}+ 
		\int\displaylimits_{\partial (B(x_0,4r)\cap\Omega)} {\alpha |v-w|^2}\\
		\le & \int\displaylimits_{B(x_0,4r)\cap \Omega} {|(B-A)\nabla v| |\nabla (v-w)|}\\
		\le & \ C_\varepsilon \int\displaylimits_{B(x_0,4r)\cap \Omega}{|(B-A)\nabla v|^2}  +\varepsilon \int\displaylimits_{B(x_0,4r)\cap \Omega}{|\nabla (v-w)|^2}.
		\end{aligned}
		\end{equation*}
		But we also have the equivalence of norm,
		\begin{equation*}
		\begin{aligned}
		\|v - w\|^2_{H^1(B(x_0,4r)\cap \Omega)} & \le C \left( \ \int\displaylimits_{B(x_0,4r)\cap \Omega}{|\nabla (v-w)|^2 } +\int \displaylimits_{B(x_0,4r)\cap \Gamma}{ |v-w|^2}\right) \\
		& \le C(\alpha_*) \left( \ \int\displaylimits_{B(x_0,4r)\cap \Omega}{|\nabla (v-w)|^2 } +\int \displaylimits_{B(x_0,4r)\cap \Gamma}{\alpha |v-w|^2}\right)
		\end{aligned}
		\end{equation*}
		where the above constant $C>0$ depends on $\Omega$ and $\alpha_*$ but is independent of $r$ and $\alpha$. 
		This gives
		\begin{align*}
		& \quad \left( \ \pppvint\displaylimits_{B(x_0,4r)\cap \Omega}{|\nabla v - \nabla w|^2 + |v-w|^2 }\right) ^{1/2}\\
		& \le C \left( \ \pppvint\displaylimits_{B(x_0,4r)\cap \Omega}{|(B-A)\nabla v|^2}\right) ^{1/2} \\
		& \le C \left( \ \pppvint\displaylimits_{B(x_0,4r)\cap \Omega}{ |\nabla v|^{2q}}\right) ^{1/2q} \left( \ \pppvint\displaylimits_{B(x_0,4r)\cap \Omega}{|B-A|^{2q' }}\right) ^{1/2q'}.
		\end{align*}
		Defining
		\begin{equation*}
		h(r) = C \underset{x_0\in \overline{\Omega}}{\sup}\ \left( \ \pppvint\displaylimits_{B(x_0,4r)\cap \Omega}{|B-A|^{2q' }}\right) ^{1/2q'}
		\end{equation*}
		the last inequality yields
		\begin{equation*}
		\begin{aligned}
		\left( \ \pppvint\displaylimits_{B(x_0,4r)\cap \Omega}{|\nabla v - \nabla w|^2 +|v-w|^2}\right) ^{1/2}
		& \le h(r) \left( \ \pppvint\displaylimits_{B(x_0,4r)\cap \Omega}{ |\nabla v|^{2q}}\right) ^{1/2q}\\
		& \le h(r) \left( \ \pppvint\displaylimits_{B(x_0,8r)\cap \Omega}{ |\nabla v|^{2}+|v|^2}\right) ^{1/2}.
		\end{aligned}
		\end{equation*}
		Note that in the last line, we used $L^{2+\varepsilon}$ weak reverse H\"{o}lder inequality (\textit{i.e.} for some $q>1$) for $v$ which follows from Lemma \ref{Lrhi}. It is known from John-Nirenberg inequality that $h(r) \to 0$ as $r\to 0$. Indeed, John-Nirenberg inequality says, for any $BMO$-function $f$,
		\begin{equation*}
		\int\displaylimits_{B}{e^{M|f-\overline{f}|}}\le C r^n
		\end{equation*}
		for some constant $C>0$ depending only on $n$. Since $A\in VMO(\R^3)$, by definition we get that $h(r)\to 0$. 
		
		Finally, to see (\ref{3}), note that $(B\nabla w)\cdot \vn + \alpha w =0$ on $B(x_0,4r)\cap\Gamma$. Thus, by Lemma \ref{Lrhi.}, we obtain, for any $p\ge 2$,
		\begin{equation*}
		\begin{aligned}
		& \quad \left( \ \ppvint\displaylimits_{B(x_0,r)\cap\Omega} {|\nabla w|^p} \right) ^{1/p}\\
		&\le C \left( \ \pppvint\displaylimits_{B(x_0,4r)\cap\Omega} {|\nabla w|^2+|w|^2} \right) ^{1/2}\\
		& \le C \left( \ \pppvint\displaylimits_{B(x_0,4r)\cap\Omega} {|\nabla v|^2+|v|^2} \right) ^{1/2} + C \left( \ \pppvint\displaylimits_{B(x_0,4r)\cap\Omega} {|\nabla (v-w)|^2+|v-w|^2} \right) ^{1/2}\\
		& \le C \left( \ \pppvint\displaylimits_{B(x_0,8r)\cap\Omega} {|\nabla v|^2 +|v|^2} \right) ^{1/2}.
		\end{aligned}
		\end{equation*}
		This shows that in fact $w\in W^{1.p}(B(x_0,r)\cap\Omega)$ which completes the proof.
		\hfill
	\end{proof}
	
	With Lemma \ref{Lapprox} at our hand, we may use the following approximation theorem, motivated from the paper of Caffarelli and Peral \cite{CP} and proved in \cite{geng}, to finish the proof of the weak reverse H\"{o}lder inequality for $VMO$ coefficient.
	
	\begin{theorem}
		\label{8}
		Let $E\subset \R^n$ be any open set and $F:E\to \R^n$ locally square integrable. Let $p>2$. Suppose there exists some constants $\beta>1$, $C>1$ and $\varepsilon>0$ such that for every cube $Q$ with $2Q = Q(x_0,2r)\subset E$, there exists a measurable function $R_Q$ on $2Q$ satisfying
		\begin{equation}
		\label{5}
		\left( \frac{1}{|Q|}\int\displaylimits_Q {|R_Q|^p} \right) ^{1/p} \le C \left( \frac{1}{| \beta Q|} \int\displaylimits_{\beta Q}{|F|^2} \right) ^{1/2}
		\end{equation}
		and
		\begin{equation}
		\label{6}
		\left( \frac{1}{|Q|}\int\displaylimits_Q {|F-R_Q|^2} \right) ^{1/2} \le \varepsilon \left( \frac{1}{|\beta Q|}\int\displaylimits_{\beta Q} {|F|^2} \right) ^{1/2}.
		\end{equation}
		Let $2<q<p$. Then, there exists $\varepsilon_0 = \varepsilon_0(C,n,p,q,\beta)$ such that if $\varepsilon<\varepsilon_0 $, we have
		\begin{equation}
		\label{7}
		\left( \frac{1}{|Q|} \int\displaylimits_{Q} {|F|^q} \right) ^{1/q} \le C_1 \left( \frac{1}{|2Q|} \int\displaylimits_{2Q} {|F|^2} \right) ^{1/2}
		\end{equation}
		where $C_1>0$ depends only on $C, n, p, q, \beta$.
	\end{theorem}
	
	\begin{theorem}
		\label{T3}
		Let $\Omega$ be a $\mathcal{C}^1$ bounded domain in $\R^3$ and $p\ge 2$. Suppose that the coefficients of operator $\mathcal{L}$, defined in (\ref{10}), are symmetric and in $VMO(\R^3)$. Then for any $B(x,r)$ with the property that $0<r<\frac{r_0}{8}$ and either $B(x,2r)\subset \Omega$ or $ x\in\Gamma$, the weak Reverse H\"{o}lder inequalities (\ref{rhi.}) and (\ref{rhi1.}) hold with constant $C>0$ independent of $\alpha$.
	\end{theorem}
	
	\begin{proof}
		Let $h(r)$ be same as in Lemma \ref{Lapprox} and choose $q$ such that $2<q<p$. Let $\varepsilon_0$ be the same as in Theorem \ref{8} and then we choose $r_0$ small enough such that $\underset{0<r<r_0}{\sup} h(r) < \varepsilon_0$.
		
		Let $v\in H^1(B(x_0,8r)\cap\Omega)$ be a weak solution of
		\begin{equation*}
		\begin{cases}
		\mathcal{L} v &= 0 \quad  \textup{ in } \ B(x_0,8r)\cap\Omega\\
		A\nabla v\cdot \vn + \alpha v &= 0 \quad \textup{ on } \ B(x_0,8r)\cap\Gamma
		\end{cases}
		\end{equation*}
		where $0<r<\frac{r_0}{8}$ and either $B(x_0,2r)\subset \Omega$ or $x_0\in \Gamma$. To apply Theorem \ref{8}, take $E = B(x_0,8r)$ and $B(x,r)$ is any ball with $B(x,2r)\subset E$. Then the proof divides in the following cases:
		
		\noindent i) if $B(x,r)\cap\Omega = \emptyset$, we take $F=0=R_B$,
		
		\noindent ii) if $B(x,r)\subset \Omega$, set $F=\nabla v$ and $R_B = \nabla u$,
		
		\noindent iii) if $B(x,r)\cap\Omega \ne \emptyset$ and $B(x,r)\cap (\overline{\Omega})^c \ne \emptyset$, we further consider the two situations: 
		
		 -- if $x\in \overline{\Omega}$, set
		\begin{equation*}
		F=(\nabla v, v) \chi_\Omega \quad \text{ and } \quad
		R_{{B}} =
		\begin{cases}
		(\nabla w, w) & \text{ on } B(x,r)\cap\Omega,\\
		0 &\text{ on } B(x,r)\cap\Omega^c;
		\end{cases}
		\end{equation*}
		
		-- if $x\notin \overline{\Omega}$, by a geometric observation, it is easy to find a ball $\tilde{B} = B(y,2r)$ such that $ y\in \Gamma$ and $B\subset \tilde{B}\subset E$, we then set
		\begin{equation*}
		F = (\nabla v, v) \chi_\Omega \quad \text{ and } \quad
		R_{\tilde{B}} =
		\begin{cases}
		(\nabla w, w) & \text{ on } B(y,2r)\cap\Omega,\\
		0 &\text{ on } B(y,2r)\cap\Omega^c.
		\end{cases}
		\end{equation*}
		The estimates (\ref{5}) and (\ref{6}) now follow from (\ref{3}) and (\ref{4}). This finishes the proof.
		
		\hfill
	\end{proof}

Now to complete the proof of Theorem \ref{T1}, we also need the following lemma which is proved in \cite[Theorem 2.2]{geng}.

\begin{lemma}
\label{G0}
Let $\Omega$ be a bounded Lipschitz domain in $\R^3$ and $p>2$. Let $G\in \L{2}$ and $f\in\L{q}$ for some $2<q<p$. Suppose that for each ball $B$ with the property that $|B|\le \beta |\Omega|$ and either $2B \subset \Omega$ or $B$ centers on $\Gamma$, there exist two integrable functions $G_B$ and $R_B$ on $2B\cap \Omega$ such that $|G|\le |G_B| + |R_B|$ on $2B\cap \Omega$ and
\begin{equation}
\label{G1}
\begin{aligned}
& \left(\frac{1}{|2B\cap \Omega|}\int\displaylimits_{2B\cap \Omega}{|R_B|^p} \right) ^{1/p}\\
& \le C_1 \left[ \left(\frac{1}{|\gamma B \cap \Omega|}\int\displaylimits_{\gamma B \cap \Omega}{|G|^2}\right)^{1/2} + \underset{B \subset B'}{sup} \left(\frac{1}{|B'\cap \Omega|}\int\displaylimits_{B'\cap \Omega}{|f|^2}\right)^{1/2}\right] 
\end{aligned}
\end{equation}
and
\begin{equation}
\label{G2}
\left(\frac{1}{|2B\cap \Omega|}\int\displaylimits_{2B\cap \Omega}{|G_B|^2} \right) ^{1/2} \le C_2 \ \underset{B \subset B'}{sup} \left(\frac{1}{|B'\cap \Omega|}\int\displaylimits_{B'\cap \Omega}{|f|^2}\right)^{1/2}
\end{equation}
where $C_1, C_2 >0$ and $0<\beta <1<\gamma$. Then we have,
\begin{equation}
\label{G3}
\left(\frac{1}{|\Omega|}\int\displaylimits_{ \Omega}{|G|^q} \right) ^{1/q} \le C \left[ \left(\frac{1}{|\Omega|}\int\displaylimits_{\Omega}{|G|^2}\right)^{1/2} + \left(\frac{1}{|\Omega|}\int\displaylimits_{\Omega}{|f|^q}\right)^{1/q}\right]
\end{equation}
where $C>0$ depends only on $C_1, C_2, n, p, q, \beta, \gamma$ and $\Omega$. 
\end{lemma}

\begin{proof}[\bf Proof of Theorem  \ref{T1}]

Given any ball $B$ with either $2B \subset \Omega$ or $B$ centers on $\Gamma$, let $\varphi\in C_c^\infty(8B)$ is a cut-off function such that $0\le \varphi \le 1$ and
\begin{equation*}
\varphi =
\begin{cases}
& 1 \quad \text{ on } 4B\\
& 0 \quad \text{ outside } 8B
\end{cases}
\end{equation*}
and we decompose $u=v+w$ where $v,w$ satisfy
\begin{equation}
\label{01}
\left\{
\begin{aligned}
\mathcal{L} v &= \div (\varphi \bm{f}) \quad &&\text{ in } \Omega\\
A\nabla v\cdot \vn + \alpha v &= \varphi \bm{f}  \cdot \vn \quad &&\text{ on } \Gamma
\end{aligned}
\right.
\end{equation}
and
\begin{equation}
\label{02}
\left\{
\begin{aligned}
\mathcal{L} w &= \div \left( (1-\varphi) \bm{f}\right)  \quad &&\text{ in } \Omega\\
A\nabla w\cdot \vn  + \alpha w &= (1-\varphi) \bm{f} \cdot \vn \quad &&\text{ on } \Gamma .
\end{aligned}
\right.
\end{equation}
Multiplying (\ref{01}) by $v$ and integrating by parts, we get,
\begin{equation*}
\int\displaylimits_\Omega{A(x)\nabla v\cdot \nabla v} + \int\displaylimits_\Gamma {\alpha |v|^2} = \int\displaylimits_\Omega{\varphi \bm{f} \cdot \nabla v}
\end{equation*}
which gives
\begin{equation}
\label{L2.}
\|\nabla v\|_{\vL{2}} \le \frac{1}{\mu} \|\varphi\bm{f}\|_{\vL{2}} .
\end{equation}
and since $\alpha \ge \alpha_* > 0$ on $\Gamma$,
\begin{equation*}
\|v\|^2_{\H1} \le C(\Omega, \alpha_*) \left( \|\nabla v\|^2_{\vL{2}} + \int\displaylimits_\Gamma{\alpha |v|^2} \right) \le C(\Omega,\alpha_*) \ \|\varphi \bm{f}\|_{\vL{2}} \|\nabla v\|_{\vL{2}}.
\end{equation*}
This yields the complete $L^2$-estimate
\begin{equation}
\label{L2}
\|v\|_{\H1} \le C(\Omega,\alpha_*) \ \|\varphi\bm{f}\|_{\vL{2}} .
\end{equation}

{\bf (i)} First we consider the case $4B\subset \Omega$. We want to apply Lemma \ref{G0} with $G = |\nabla u|, G_B = |\nabla
 v|$ and $ R_B = |\nabla w|$. It is easy to see that
$$
|G| \le |G_B| + |R_B| .
$$
Now we verify (\ref{G1}) and (\ref{G2}). For that, using (\ref{L2.}) we get,
\begin{equation*}
\begin{aligned}
\frac{1}{|2B|} \int\displaylimits_{2B}{|G_B|^2} = \frac{1}{|2B|} \int\displaylimits_{2B}{|\nabla v|^2} \le \frac{1}{| 2B\cap\Omega|} \int\displaylimits_{ \Omega}{|\nabla v|^2}
& \le \frac{C(\Omega,\alpha_*)}{|2B\cap \Omega|} \int\displaylimits_{\Omega}{|\varphi \bm{f}|^2} \\
& \le \frac{C(\Omega,\alpha_*)}{|8B\cap \Omega|} \int\displaylimits_{8B \cap\Omega}{| \bm{f}|^2}
\end{aligned}
\end{equation*}
where in the last inequality, we used that $|8B\cap\Omega| \le |\Omega|$. This gives the estimate (\ref{G2}).

Next, from (\ref{02}), we observe that $\mathcal{L} w = 0 \text{ in } 4B$. Hence, by the weak reverse H\"{o}lder inequality in Theorem \ref{T3} (using $2B$ instead of $B$), we have
\begin{equation*}
\left( \frac{1}{|2B|} \int\displaylimits_{2B}{|\nabla w|^p}\right) ^{1/p} \le C \left( \frac{1}{|4B|}\int\displaylimits_{4B }{|\nabla w|^2}\right) ^{1/2} 
\end{equation*}
which implies together with (\ref{L2.}),
\begin{equation*}
\begin{aligned}
\left( \frac{1}{|2B|} \int\displaylimits_{2B}{|R_B|^p}\right) ^{1/p}
&\le C \left( \frac{1}{|4B|}\int\displaylimits_{4B }{|\nabla w|^2}\right) ^{1/2} \\
& \le C \left[ \left( \frac{1}{|4B|}\int\displaylimits_{4B }{|\nabla u|^2}\right) ^{1/2} + \left( \frac{1}{|4B|}\int\displaylimits_{4B }{|\nabla v|^2}\right) ^{1/2} \right] \\
& \le C \left( \frac{1}{|4B|}\int\displaylimits_{4B }{|G|^2}\right) ^{1/2} + C(\Omega,\alpha_*) \left( \frac{1}{|8B\cap \Omega|}\int\displaylimits_{8B \cap\Omega}{|\bm{f}|^2}\right) ^{1/2}.
\end{aligned}
\end{equation*}
This gives (\ref{G1}). So from Lemma \ref{G0}, it follows that
\begin{equation*}
\left(\frac{1}{|\Omega|}\int\displaylimits_{ \Omega}{|\nabla u|^q} \right) ^{1/q} \le C_p(\Omega) \left[ \left(\frac{1}{|\Omega|}\int\displaylimits_{\Omega}{|\nabla u|^2}\right)^{1/2} + \left(\frac{1}{|\Omega|}\int\displaylimits_{\Omega}{|\bm{f}|^q}\right)^{1/q}\right]
\end{equation*}
for any $2<q<p$ where $C_p(\Omega)>0$ does not depend on $\alpha$.

Because of the self-improving property of the weak Reverse H\"{o}lder condition (\ref{rhi}), the above estimate holds for any $q\in(2,\tilde{p})$ for some $\tilde{p}>p$ also and in particular, for $q=p$, which clearly implies (\ref{Lp}).\\

{\bf (ii)} Next consider $B$ centers on $\Gamma$. We apply Lemma \ref{G0} now with $G=|u| + |\nabla u|, G_B = |v|+|\nabla v|$ and $ R_B=|w|+|\nabla w|$. Obviously,
$
|G| \le |G_B| + |R_B|
$
and again by (\ref{L2}),
\begin{equation*}
\begin{aligned}
\frac{1}{|2B\cap \Omega|} \int\displaylimits_{2B\cap \Omega}{|G_B|^2} \le \frac{1}{|2B\cap \Omega|} \int\displaylimits_{2B\cap \Omega}{(|v|^2 +|\nabla v|^2)} &\le \frac{1}{|2B\cap \Omega|} \ \| v\|^2_{\H1} \\
&\le \frac{C(\Omega, \alpha_*)}{|2B\cap \Omega|} \int\displaylimits_{\Omega}{|\varphi \bm{f}|^2} \\
&\le \frac{C(\Omega,\alpha_*)}{|8B\cap \Omega|} \int\displaylimits_{8B \cap\Omega}{| \bm{f}|^2}
\end{aligned}
\end{equation*}
which yields (\ref{G2}). Also $w$ satisfies the problem
\begin{equation*}
\left\{
\begin{aligned}
\mathcal{L} w &= 0 \quad &&\text{ in } 4B \cap \Omega\\
A\nabla w\cdot \vn  + \alpha w &= 0 \quad &&\text{ on } 4B \cap\Gamma.
\end{aligned}
\right.
\end{equation*}
So by the weak reverse H\"{o}lder inequality in Theorem \ref{T3} and the estimate (\ref{L2.}), we can write,
\begin{equation*}
\begin{aligned}
&\left( \frac{1}{|2B\cap \Omega|} \int\displaylimits_{2B\cap \Omega}{|R_B|^p}\right) ^{1/p}\\
&\le \left( \frac{1}{|2B\cap \Omega|}\int\displaylimits_{2B \cap\Omega}{((| w|+ | \nabla w|)^2)^{p/2}}\right) ^{1/p} \\
&\le C \left( \frac{1}{|4B\cap \Omega|}\int\displaylimits_{4B \cap\Omega}{(| w|^2+ | \nabla w|^2)}\right) ^{1/2}\\
& \le C \left[ \left( \frac{1}{|4B\cap \Omega|}\int\displaylimits_{4B \cap\Omega}{(|u|^2+|\nabla u|^2)}\right) ^{1/2} + \left( \frac{1}{|4B\cap \Omega|}\int\displaylimits_{4B \cap\Omega}{(|v|^2+|\nabla v|^2)}\right) ^{1/2} \right] \\
& \le C \left( \frac{1}{|4B\cap \Omega|}\int\displaylimits_{4B \cap\Omega}{|G|^2}\right) ^{1/2} + C(\Omega,\alpha_*) \left( \frac{1}{|8B\cap \Omega|}\int\displaylimits_{8B \cap\Omega}{|\bm{f}|^2}\right) ^{1/2}
\end{aligned}
\end{equation*}
which yields (\ref{G1}). Thus we have,
$$
\left(\frac{1}{|\Omega|}\int\displaylimits_{ \Omega}{(|u|+|\nabla u|)^q} \right) ^{1/q} \le C_p(\Omega, \alpha_*) \left[ \left(\frac{1}{|\Omega|}\int\displaylimits_{\Omega}{|u|^2+|\nabla u|^2}\right)^{1/2} + \left(\frac{1}{|\Omega|}\int\displaylimits_{\Omega}{|\bm{f}|^q}\right)^{1/q}\right]
$$
for any $2<q<p$ where $C_p(\Omega,\alpha_*)>0$ does not depend on $\alpha$. This completes the proof together with the previous case.
\hfill
\end{proof}

The next proposition will be used to study the complete estimate of the Robin problem (\ref{robin1}). The result is not optimal and will be improved in Proposition \ref{P2}.

\begin{proposition}[{\bf \boldmath$\W{1}{p}$ estimate, $p>2$ with RHS $F$}]
\label{P0}
Let $\Omega$ be a $\mathcal{C}^1$ bounded domain in $\R^3$, $p>2$, and $F\in\L{p}$. Suppose that the coefficients of the operator $\mathcal{L}$, defined in (\ref{10}), are also symmetric and in $VMO(\R^3)$. Then the unique solution $u\in\W{1}{p}$ of (\ref{robin1}) with $\bm{f}=\bm{0}$ and $g=0$, satisfies the following estimate:
\begin{equation}
\label{Lp0}
\|u\|_{\W{1}{p}} \le C_p(\Omega,\alpha_*) \ \|F \|_{\L{p}} 
\end{equation}
where the constant $C_p(\Omega,\alpha_*)>0$ is independent of $\alpha$.
\end{proposition}

\begin{proof}
The result follows using the same argument as in Theorem \ref{T1} and hence we do not repeat it.
\hfill
\end{proof}

\begin{proposition}[{\bf \boldmath$\W{1}{p}$ estimate with RHS $\bm{f}$}]
	\label{P1}
	Let $\Omega$ be a $\mathcal{C}^1$ bounded domain in $\R^3$, $p\in(1,\infty)$ and $\bm{f}\in\vL{p}$. Suppose that the coefficients of the operator $\mathcal{L}$, defined in (\ref{10}), are also symmetric and in $VMO(\R^3)$. Then there exists a unique solution $u\in\W{1}{p}$ of (\ref{robin1}) with $F=0$ and $g=0$, satisfying the following estimate:
	\begin{equation}
	\label{Lp2}
	\|u\|_{\W{1}{p}} \le C_p(\Omega,\alpha_*) \ \|\bm{f} \|_{\vL{p}} 
	\end{equation}
	where the constant $C_p(\Omega,\alpha_*)>0$ is independent of $\alpha$.
\end{proposition}

\begin{proof}
The existence of a unique solution and the corresponding estimate for $p>2$ is done in Theorem \ref{T0} and Theorem \ref{T1} respectively. Now suppose that $1<p<2$. We first discuss the estimate and then the existence of a solution.\\
\\
{\bf (i)} {\bf Estimate I:} Let $\bm{g}\in C_0^\infty (\Omega)$ and $v\in\W{1}{p'}$ be the solution of $\mathcal{L} v = \div \ \bm{g}$ in $\Omega$ and $\frac{\partial v}{\partial \vn} + \alpha v = 0$ on $\Gamma$. Since $p'>2$, from Theorem \ref{T1}, we have
	$$
	\|v\|_{\W{1}{p'}} \le C_p(\Omega,\alpha_*) \|\bm{g}\|_{\vL{p'}} .
	$$
	Also if $u\in\W{1}{p}$ is a solution of (\ref{robin1}) with $F=0, g=0$, using the weak formulation of the problems satisfied by $u$ and $v$, we have
	$$
	\int\displaylimits_\Omega{\bm{f}\cdot \nabla v} = 
	\int\displaylimits_\Omega{\bm{g} \cdot \nabla u}
	$$
	which gives,
	$$
	|\int\displaylimits_\Omega{\bm{g} \cdot \nabla u}| \le \|\bm{f}\|_{\vL{p}} \|\nabla v\|_{\vL{p'}} \le \|\bm{f}\|_{\vL{p}} \| v\|_{\W{1}{p'}}
	$$
	and hence,
	$$
	\|\nabla u\|_{\vL{p}} = \underset{0\neq \bm{g}\in\vL{p'}}{\sup}\frac{|\int\displaylimits_\Omega{\nabla u\cdot \bm{g} }|}{\|\bm{g}\|_{\vL{p'}}} \le C_p(\Omega, \alpha_*) \|\bm{f}\|_{\vL{p}}.
	$$
{\bf (ii)} {\bf Estimate II:} Next we prove that 
\begin{equation}
\label{Lp2.}
\|u\|_{\L{p}} \le C_p(\Omega, \alpha_*) \|\bm{f}\|_{\vL{p}}.
\end{equation}
For that, from Proposition \ref{P0}, we get for any $\varphi\in\L{p'}$, the unique solution $w\in \W{1}{p'}$ of the problem
\begin{equation*}
\left\{
\begin{aligned}
\mathcal{L} w &= \varphi \quad &&\text{ in } \Omega\\
A\nabla w\cdot \vn  + \alpha w &= 0  \quad &&\text{ on }\Gamma
\end{aligned}
\right.
\end{equation*}
satisfies
$$
\|w\|_{\W{1}{p'}}\le C_p(\Omega,\alpha_*) \ \|\varphi\|_{\vL{p'}}.
$$
Therefore using the weak formulation of the problems satisfied by $u$ and $w$, we obtain,
\begin{align*}
\int\displaylimits_\Omega{u \ \varphi} = \int\displaylimits_\Omega{ u \, \div (A(x)\nabla) w} = -\int\displaylimits_\Omega{A(x)\nabla u \cdot \nabla w} + \int\displaylimits_\Gamma{u A\nabla w\cdot \vn}  = - \int\displaylimits_\Omega{\bm{f} \cdot \nabla w}
\end{align*}
which implies
\begin{equation*}
\|u\|_{\L{p}} = \underset{0\neq \varphi \in\L{p'}}{\sup}\frac{|\int\displaylimits_\Omega{ u \ \varphi }|}{\|\varphi\|_{\L{p'}}}\le C_p(\Omega,\alpha_*) \ \|\bm{f}\|_{\vL{p}}.
\end{equation*}
This completes proof of the estimate (\ref{Lp2}).\\
\\
{\bf (iii)} {\bf Existence and uniqueness:} The uniqueness of solution of (\ref{robin1}) follows from (\ref{Lp2}). For the existence, we will use a limit argument. Let $\{\bm{f}_k\} \in C_0^\infty (\Omega)$ such that
$$ \bm{f}_k \rightarrow \bm{f} \quad \text{ in } \vL{p} $$
and $u_k \in \W{1}{p'}$ be the unique solution of 
\begin{equation} \label{1}
\left\{
\begin{aligned}
\mathcal{L} u_k &= \div \ \bm{f_k} \quad &&\text{ in } \Omega\\
(A\nabla u_k - \bm{f_k})\cdot \vn  + \alpha u_k &= 0 \quad &&\text{ on } \Gamma
\end{aligned}
\right.
\end{equation}
Note that $u_k \in \W{1}{p}$ since $p'>2$. Also from {\bf (i)} we have,
$$
\| u_k\|_{\W{1}{p}} \le C_p(\Omega,\alpha_*) \ \|\bm{f}_k\|_{\vL
p}
$$
and
$$
\| u_k-u_{\ell}\|_{\W{1}{p}} \le C_p(\Omega,\alpha_*) \ \|\bm{f}_k-\bm{f}_{\ell}\|_{\vL p}.
$$
Thus it follows $u_k-u_{\ell} \rightarrow 0$ in $\W{1}{p}$ as $k,{\ell}\rightarrow \infty$ i.e. $\{u_k\}$ is a Cauchy sequence in $\W{1}{p}$. Then as $\W{1}{p}$ is a Banach space, there exists $u\in\W{1}{p}$ such that
$$u_k\rightarrow u\text{ in } \W{1}{p}$$
satisfying
$$\|u\|_{\W{1}{p}}\le C_p(\Omega,\alpha_*) \ \|\bm{f}\|_{\vL{p}}.$$
Clearly $u$ also solves the system (\ref{robin1}).
\hfill
\end{proof}

\begin{proposition}[{\bf \boldmath$\W{1}{p}$ estimate with RHS $F$}]
	\label{P2}
	Let $\Omega$ be a $\mathcal{C}^1$ bounded domain in $\R^3$, $p\in(1,\infty)$, $F\in \L{r(p)}$ and $g\in \Wfracb{-\frac{1}{p}}{p}$. Suppose that the coefficients of the operator $\mathcal{L}$, defined in (\ref{10}), are symmetric and in $VMO(\R^3)$. Then the solution $u\in\W{1}{p}$ of the problem
	\begin{equation}
	\label{robin2}
	\left\{
	\begin{aligned}
	\mathcal{L} u &= F \quad &&\text{ in } \Omega\\
	A\nabla u\cdot \vn  + \alpha u &= g \quad &&\text{ on } \Gamma
	\end{aligned}
	\right.
	\end{equation}
satisfies the following estimate:
	\begin{equation}
	\label{Lp3}
	\|u\|_{\W{1}{p}} \le C_p(\Omega,\alpha_*) \left(  \|F \|_{\L{r(p)}} + \|g\|_{\Wfracb{-\frac{1}{p}}{p}}\right) 
	\end{equation}
	where the constant $C_p(\Omega,\alpha_*)>0$ is independent of $\alpha$.
\end{proposition}

\begin{proof}
It suffices to prove the estimate since the existence and uniqueness of $u$ follows from the same argument as in Proposition \ref{P1}.

{\bf(i)} {\bf Estimate I:} Let $\bm{f}\in C_0^\infty (\Omega)$ and $v\in\W{1}{p'}$ be the weak solution of $\mathcal{L} v = \div \ \bm{f}$ in $\Omega$ and $(A\nabla v - \bm{f} )\cdot \vn + \alpha v = 0$ on $\Gamma$. By Proposition \ref{P1}, we then have
	$$
	\|v\|_{\W{1}{p'}} \le C_p(\Omega,\alpha_*) \|\bm{f}\|_{\vL{p'}} .
	$$
Also, if $u\in\W{1}{p}$ is a solution of (\ref{robin2}), from the weak formulation of the problems satisfied by $u$ and $v$, we get
	$$
	\int\displaylimits_{\Omega}{\bm{f}\cdot \nabla u} = \int\displaylimits_{\Omega}{A(x)\nabla u\cdot \nabla v} + \int\displaylimits_{\Gamma}{\alpha u v} = - \int\displaylimits_{\Omega}{F v}+\left\langle g ,v\right\rangle _{\Gamma} .
	$$
This implies
	\begin{equation*}
	\begin{aligned}
	|\int\displaylimits_\Omega{\bm{f}\cdot \nabla u}| &\le \|F\|_{\L{r(p)}} \|v\|_{\L{(r(p))'}} + \|g\|_{\Wfracb{-\frac{1}{p}}{p}} \|v\|_{\Wfracb{\frac{1}{p}}{p'}}\\
	&\le C_p(\Omega) \left( \|F\|_{\L{r(p)}}+ \|g\|_{\Wfracb{-\frac{1}{p}}{p}}\right)  \|v\|_{\W{1}{p'}}
	\end{aligned}
	\end{equation*}
	since $\frac{1}{(p')*} = \frac{1}{p'}-\frac{1}{3} = \frac{1}{(r(p))'}$ for $p> \frac{3}{2}$ and $\W{1}{p'}\hookrightarrow \L{\infty}$ when $p<\frac{3}{2}$. Thus, 
	\begin{equation*}
	\begin{aligned}
	\|\nabla u\|_{\vL{p}} = \underset{0\neq \bm{f}\in \vL{p'}}{\sup}\frac{\left| \int\displaylimits_{\Omega}{\nabla u\cdot \bm{f}}\right| }{\|\bm{f}\|_{\vL{p'}}} \le  C_p(\Omega,\alpha_*) \left( \|F\|_{\L{r(p)}}+ \|g\|_{\Wfracb{-\frac{1}{p}}{p}}\right)  .
	\end{aligned}
	\end{equation*}
{\bf (ii)} {\bf Estimate II:} Next we prove the following bound as done in (\ref{Lp2.}):
\begin{equation}
\label{Lp5}
\|u\|_{\L{p}} \le C_p(\Omega,\alpha_*) \left( \|F\|_{\L{r(p)}}+ \|g\|_{\Wfracb{-\frac{1}{p}}{p}}\right)
\end{equation}
except that we do not need to assume $p<2$ here as in (\ref{Lp2.}). 
For any $\varphi\in\L{p'}$, there exists a unique $w\in\W{1}{p'}$ solving the problem
\begin{equation*}
\left\{
\begin{aligned}
\mathcal{L} w &= \varphi \quad &&\text{ in } \Omega\\
A\nabla w\cdot \vn + \alpha w &= 0  \quad &&\text{ on }\Gamma
\end{aligned}
\right.
\end{equation*}
and satisfying
$$
\|w\|_{\W{1}{p'}} \le C_p(\Omega,\alpha_*) \|\varphi\|_{\L{p'}}.
$$
(For $p < 2$ the above estimate can be proved by the exact same argument as in Proposition (\ref{P1})). Finally we may write,
\begin{equation*}
\int\displaylimits_{\Omega}{u \ \varphi} =\int\displaylimits_{\Omega}{u\,\,\mathcal{L} w} = \int\displaylimits_{\Omega}{\mathcal{L} u \ w } - \int\displaylimits_{\Gamma}{(A\nabla u\cdot \vn) w} + \int\displaylimits_{\Gamma}{u(A\nabla w)\cdot \vn } =  \int\displaylimits_{\Omega}{F w } - \left\langle g, w\right\rangle _{\Gamma} 
\end{equation*}
which yields as before
$$
\|u\|_{\L{p}} \le C_p(\Omega,\alpha_*) \left( \|F\|_{\L{r(p)}}+ \|g\|_{\Wfracb{-\frac{1}{p}}{p}}\right)
$$
and thus we obtain (\ref{Lp5}).
\hfill
\end{proof}

\begin{proof}[\bf Proof of Theorem  \ref{T2}]
Let $u_1\in\W{1}{p}$ be the weak solution of
\begin{equation*}
\left\{
\begin{aligned}
\div (A(x)\nabla u_1) &= \div \bm{f} \quad  \text{ in } \Omega\\
(A\nabla u_1 - \bm{f})\cdot \vn + \alpha u_1 &=   0 \quad \text{ on } \Gamma
\end{aligned}
\right.
\end{equation*}
given by Proposition \ref{P1} and $u_2\in\W{1}{p}$ be the weak solution of
\begin{equation*}
\left\{
\begin{aligned}
\div (A\nabla u_2) &=  F \quad  \text{ in } \Omega\\
A\nabla u_2\cdot \vn + \alpha u_2 &= g \quad \text{ on } \Gamma
\end{aligned}
\right.
\end{equation*}
given by Proposition \ref{P2}. Then $u=u_1+u_2$ is the solution of the problem (\ref{robin1}) which also satisfies the estimate (\ref{Lp1}).
\hfill
\end{proof}

Next we prove uniform $H^s$ bound for $s\in(0,\frac{1}{2})$.
\begin{proposition}
\label{P4}
Let $\Omega$ be a Lipschitz bounded domain in $\R^{3}$, $g\in \Lb{2}$ and $\alpha$ is a constant. Suppose that the coefficients of the operator $\mathcal{L}$, defined in (\ref{10}), are symmetric and in $VMO(\R^3)$. Then the problem
\begin{equation}
\label{robin3}
\left\{
\begin{aligned}
&\mathcal{L} u = 0 \quad &&\text{ in } \Omega\\
&A\nabla u\cdot \vn + \alpha u = g \quad &&\text{ on } \Gamma
\end{aligned}
\right.
\end{equation}
has a solution $u\in \H{3/2}$ which also satisfies the estimate
\begin{equation}
\|u\|_{\H{3/2}} \le C(\Omega) \|g\|_{\Lb{2}}.
\end{equation}
\end{proposition}

\begin{proof}
A solution $u\in\H1$ of the problem (\ref{robin3}) satisfies the variational formulation:
\begin{equation*}
\forall \varphi \in\H{1}, \qquad \int\displaylimits_{ \Omega}{A(x)\nabla u\cdot \nabla \varphi} + \int\displaylimits_{\Gamma}{\alpha u\varphi} = \int\displaylimits_{\Gamma}{g\varphi}.
\end{equation*}
Multiplying the above relation by $\alpha$ and substituting $\varphi =u$, we get
\begin{equation*}
\alpha \int\displaylimits_{ \Omega}{A(x)\nabla u\cdot \nabla u} + \|\alpha u\|^2_{\Lb{2}} = \alpha \int\displaylimits_{\Gamma}{g u}\le  \|g\|_{\Lb{2}} \|\alpha u\|_{\Lb{2}}
\end{equation*}
and thus
\begin{equation*}
\|\alpha u\|_{\Lb{2}}\le \|g\|_{\Lb{2}}.
\end{equation*}
Now from the regularity result for Neumann problem \cite[Theorem 2]{kenig-jerison}, we obtain
\begin{equation*}
\|u\|_{\H{\frac{3}{2}}} \le C(\Omega) \|g-\alpha u\|_{\Lb{2}} \le C(\Omega) \|g\|_{\Lb{2}}
\end{equation*}
which gives the required estimate.
\hfill
\end{proof}

\begin{theorem}[{\bf \boldmath$H^s(\Omega)$ estimate}]
\label{Hs}
Let $\Omega$ be a Lipschitz bounded domain in $\R^{3}$, $s\in (0,\frac{1}{2})$ and $\alpha$ is a constant. Then for $g\in H^{s-\frac{1}{2}}(\Gamma)$, the problem (\ref{robin3}) has a solution $u\in \H{1+s}$ which also satisfies the estimate
\begin{equation*}
\|u\|_{\H{1+s}} \le C(\Omega) \|g\|_{H^{s-\frac{1}{2}}(\Gamma)}.
\end{equation*}
\end{theorem}

\begin{proof}
We obtain the result by interpolation between $\H1$ and $\H{\frac{3}{2}}$ regularity results in Theorem \ref{T0} and Proposition \ref{P4} respectively.
\hfill
\end{proof}

\section{Estimate for strong solution}
\setcounter{equation}{0}

\begin{theorem}[{\bf \boldmath$\W{2}{p}$ estimate}]
\label{strong}
	Let $\Omega$ be a $\mathcal{C}^{1,1}$ bounded domain in $\R^3$, $p\in (1,\infty)$ and $\alpha$ be a constant. Then for $F\in \L{p}$ and $g\in \Wfracb{1-\frac{1}{p}}{p}$, the solution $u$ of the problem
	\begin{equation}
	\label{robin}
	\left\{
	\begin{aligned}
	\Delta u &= F \quad &&\text{ in } \Omega,\\
	\frac{\partial u}{\partial \vn}+ \alpha u  &= g \quad &&\text{ on } \Gamma
	\end{aligned}
	\right.
	\end{equation}
belongs to $\W{2}{p}$ and satisfies the following estimate:
	\begin{equation}
	\label{W2pE}
	\|u\|_{\W{2}{p}} \le C_p(\Omega,\alpha_*) \left( \|F\|_{\L{p}} + \|g\|_{\Wfracb{1-\frac{1}{p}}{p}}\right) 
	\end{equation}
	where the constant $C_p(\Omega,\alpha_*)>0$ is independent of $\alpha$.
\end{theorem}

\begin{remark}
	\rm{
	 We can in fact show the existence of $u\in \W{2}{p}$ for more general $\alpha$, not necessarily constant; in particular for $\alpha \in \Wfracb{1-\frac{1}{q}}{q}$ with $q>\frac{3}{2}$ if $p\le \frac{3}{2}$ and $q=p$ otherwise.
	}
\end{remark}

\begin{proof}
For the given data, there exists a unique solution $u$ of (\ref{robin}) in $\W{1}{p}$, by Theorem \ref{T2}. Then it can be shown that in fact $u$ belongs to $\W{2}{p}$ by Neumann regularity result using bootstrap argument. But concerning the estimate, we do not obtain a $\alpha$ independent bound on $u$, using the estimate for Neumann problem. So we consider the following argument. 
	
As $\Gamma$ is compact and of class $\HC{1}{1}$, there exists an open cover $U_i$ \textit{i.e.} $\Gamma \subset \cup_{i=1}^k U_i$ and bijective maps $H_i: Q\rightarrow U_i$ such that
$$
H_i\in \mathcal{C}^{1,1}(\overline{Q}), \ J^i:=H_i^{-1}\in \mathcal{C}^{1,1}(\overline{U_i}), \ H_i(Q_+) = \Omega \cap U_i \ \text{ and } \ H_i(Q_0) = \Gamma \cap U_i
$$
where we denote
\begin{equation*}
\begin{aligned}
& Q = \{x= (x', x_3); |x'|< 1 \text{ and } |x_3|<1 \}\\
& Q_+ = Q\cap \R^3_+\\
& Q_0 = \{x=(x', 0); |x'| <1\} .
\end{aligned}
\end{equation*}
Then we consider the partition of unity $\theta_i$ corresponding to $U_i$ with $\text{supp} \ \theta_i \subset U_i$. So we can write $u = \sum_{i=0}^k \theta_i u$ where $\theta_0\in C_c^\infty(\Omega)$.
It is easy to see that $v_i = \theta_i u\in W^{2,p}(\Omega\cap U_i)$ and satisfies:
\begin{equation*}
\left\{
\begin{aligned}
 \Delta v_i = \theta_i F + 2 \nabla \theta_i \nabla u + u \Delta \theta_i =: f_i \quad & \text{ in } \quad \Omega \cap U_i\\
 \frac{\partial v_i}{\partial \vn} + \alpha v_i = g + \frac{\partial \theta_i}{\partial \vn} u =: h_i \quad & \text{ on } \quad \partial(\Omega \cap U_i) .
\end{aligned}
\right.
\end{equation*}
Precisely, we have, for all $\varphi \in W^{1,p'}(\Omega\cap U_i)$,
\begin{equation}
\label{23}
\int\displaylimits_{ \Omega\cap U_i} {\nabla v_i \cdot \nabla \varphi} + \alpha \int\displaylimits_{ \Gamma\cap U_i} {v_i \varphi}= -\int\displaylimits_{ \Omega\cap U_i} {f_i \varphi} + \int\displaylimits_{ \Gamma\cap U_i}{h_i \varphi}
\end{equation}
where $f_i\in L^p(\Omega)$ and $h_i\in W^{1-\frac{1}{p},p}(\Gamma)$. Now to transfer $v_i|_{\Omega \cap U_i} $ to $Q_+$, set $w_i(y) = v_i(H_i(y))$ for $y\in Q_+$. Then,
$$
\frac{\partial v_i}{\partial x_j} = \sum_k \frac{\partial w_i}{\partial y_k} \frac{\partial J^i_k}{\partial x_j} .
$$
Also let $\psi \in H^1(Q_+)$ and set $\varphi (x) = \psi(J^i(x))$ for $x\in \Omega\cap U_i$. Then $\varphi \in H^1(\Omega\cap U_i)$ and
$$
\frac{\partial \varphi }{\partial x_j} = \sum_l \frac{\partial \psi}{\partial y_l} \frac{\partial J^i_l}{\partial x_j}.
$$
Thus, putting these in (\ref{23}), we obtain under this change of variable, for all $\psi\in H^1(Q_+)$,
\begin{equation}
\label{24}
\int\displaylimits_{Q_+}{a_{kl}(x) \frac{\partial w_i}{\partial y_k} \frac{\partial \psi}{\partial y_l}} + \alpha \int\displaylimits_{Q_0}{w_i \psi} = - \int\displaylimits_{Q_+}{\tilde{f}_i\psi} + \int\displaylimits_{Q_0}{\tilde{h}_i\psi}
\end{equation}	
with $a_{kl}(x) = \sum_j \frac{\partial J^i_k}{\partial x_j} \frac{\partial J^i_l}{\partial x_j} |\det Jac \ H_i|$, $\tilde{f}_i = f_i\circ J^i$ and $\tilde{h}_i = h_i\circ J^i$. Here $\det Jac \ H_i$ denotes the determinant of the Jacobian matrix of $H_i$. Note that $a_{kl}\in \mathcal{C}^{0,1}(\overline{Q_+})$, $\tilde{f}_i\in L^p (Q_+)$ and $\tilde{h}_i\in W^{1/p',p}(Q_0)$. Also (\ref{24}) is a Robin problem of the form (\ref{robin1}) for $w_i$ on $Q_+$, since $w_i$ vanishes in a neighbourhood of $\partial Q_+ \smallsetminus Q_0$. 

For notational convenience, in this last part, we omit the index $i$ \textit{i.e.} we simply write $w$ instead of $w_i$. Now denoting $\partial _j = \frac{\partial}{\partial x_j}$, we see that $z_i := \partial_i w, i =1, 2$ solves the following problem
\begin{equation}
\left\{
\begin{aligned}
\div (A\nabla  z_i) &= \div (\tilde{f} \bm{e}_i) - \div (\partial_i A\nabla w)
\quad &&\text{ in } Q_+\\
(A\nabla z_i - \tilde{f} \bm{e}_i)\cdot \vn + \alpha z_i &=  -(\partial_i A\nabla w)\cdot \vn + \partial_i \tilde{h} \quad &&\text{ on } Q_0 
\end{aligned}
\right.
\end{equation}
where $\bm{e}_i$ is the unit vector with $1$ in $i^{th}$ position
Thus, we can apply Theorem \ref{T2} for the above system and may conclude
\begin{equation*}
\|z_i\|_{W^{1,p}(Q_+)} \le C_p(Q_+) \left( \|\tilde{f}\|_{L^p(Q_+)} + \|\partial_i A(x)\nabla w\|_{{\bm L}^p(Q_+)} +\|\partial_i \tilde{h}\|_{W^{-\frac{1}{p},p}(Q_0)}\right) 
\end{equation*}
which yields, $\text{ for all } i, j = 1, 2, 3 \text{ except } i=j=3$,
\begin{equation}
\label{25}
\|\partial ^2_{ij} w\|_{L^p (Q_+)} \le C_p (Q_+) \left( \|\tilde{f}\|_{L^p(Q_+)} + \|w\|_{W^{1,p}(Q_+)} +\| \tilde{h}\|_{W^{\frac{1}{p'},p}(Q_0)}\right) . 
\end{equation}
Now to show the estimate for $\partial^2_{33} w$, we can write from the equation (\ref{24}) (omitting the index $i$),
\begin{equation*}
\partial^2_{33} w = \frac{1}{a_{33}} \left( \tilde{f} - a_{ij} \ \partial ^2_{ij} w - \partial_i a_{ij} \ \partial_j w \right) \quad \text{ in } \quad Q_+ .
\end{equation*}
But since $J$ is an one-one map, $a_{33}\ne 0$ and thus together with (\ref{25}), we obtain the same estimate (\ref{25}) for $\partial^2_{33} w$.
Therefore, we can conclude, for all $i= 1, ..., k$,
\begin{equation*}
\|v\|_{W^{2,p}(\Omega\cap U_i)} \le C_p(\Omega) \left(\|F\|_{\L{p}} + \|g\|_{ \Wfracb{1-\frac{1}{p}}{p}}+ \|u\|_{\W{1}{p}} \right) 
\end{equation*}
and consequently (\ref{W2pE}), using $W^{1,p}$-estimate result.
\hfill
\end{proof}	
	
\noindent\textbf{Acknowledgement}: The authors would like to thank Karthik Adimurthi for his valuable suggestions regarding the initial write up of this manuscript.

\bibliographystyle{plain}
\bibliography{Robin_bib}

\end{document}